\documentclass[12pt,a4paper]{amsart}
\usepackage{amsfonts}
\usepackage{amssymb}
\usepackage{amsmath}
\usepackage{amsthm}
\usepackage{cite}
\usepackage{graphicx}
\usepackage{stmaryrd}
\usepackage{amscd}
\usepackage{color}
\usepackage{mathtools} 
\usepackage{cancel}
\newtheorem{theorem}{Theorem}
\theoremstyle{plain}

\newtheorem{corollary}[theorem]{Corollary}
\newtheorem{lemma}[theorem]{Lemma}

\numberwithin{equation}{section}

\newcommand{\R}{\mathbb{R}}
\newcommand{\C}{\mathbb{C}}

\newcommand{\mS}{\mathbb{S}}

\renewcommand{\phi}{\varphi}

\newcommand{\bpm}{{\begin{pmatrix}}}
\renewcommand{\phi}{\varphi}

\DeclareMathOperator{\divergenz}{div}

\DeclareMathOperator{\const}{const.}

\DeclareMathOperator{\modd}{mod}

\def\blem{\begin{lemma}}\def\elem{\end{lemma}}
\def\bthm{\begin{theorem}}\def\ethm{\end{theorem}}
\def\bcor{\begin{corollary}}\def\ecor{\end{corollary}}
\def\beq{\begin{equation}}\def\eeq{\end{equation}}

\newtagform{pm}{(}{)\rlap{$_\pm$}}

\begin{document}

\title[A breather construction for a semilinear curl-curl wave equation]{A breather construction for a semilinear curl-curl wave equation with radially symmetric coefficients}

\author{Michael Plum}
\address{M. Plum \hfill\break 
Institute for Analysis, Karlsruhe Institute of Technology (KIT), \hfill\break
D-76128 Karlsruhe, Germany}
\email{michael.plum@kit.edu}

\author{Wolfgang Reichel}
\address{W. Reichel \hfill\break 
Institute for Analysis, Karlsruhe Institute of Technology (KIT), \hfill\break
D-76128 Karlsruhe, Germany}
\email{wolfgang.reichel@kit.edu}

\date{\today}


\begin{abstract} We consider the semilinear curl-curl wave equation 
$s(x) \partial_t^2 U  +\nabla\times\nabla\times U + q(x) U \pm V(x) |U|^{p-1} U = 0 \mbox{ for } (x,t)\in \R^3\times\R$. For any $p>1$ we prove the existence of time-periodic spatially localized real-valued solutions (breathers) both for the $+$ and the $-$ case under slightly different hypotheses. Our solutions are classical solutions that are radially symmetric in space and decay exponentially to $0$ as $|x|\to \infty$. Our method is based on the fact that gradient fields of radially symmetric functions are annihilated by the curl-curl operator. Consequently, the semilinear wave equation is reduced to an ODE with $r=|x|$ as a parameter. This ODE can be efficiently analyzed in phase space. As a side effect of our analysis, we obtain not only one but a full continuum of phase-shifted breathers $U(x,t+a(x))$, where $U$ is a particular breather and $a:\R^3\to\R$ an arbitrary radially symmetric $C^2$-function. 
\end{abstract}
\maketitle

%%%%%%%%%%%%%%%%%%%%%%%%%%%%%%%%%%%%%%%%%%%%%%%%%%%%%%%%%%%%%%%%%%%%%%%%%%%

\section{Introduction}
Real-valued breathers (i.e., time-periodic spatially localized solutions) of nonlinear wave equations in $\R^d\times\R$ have attracted attention of both physicists and mathematicians with the sine-Gordon equation \eqref{sine_gordon} being a prominent example. The phenomenon of existence of breathers is quite rare. In the context of semilinear scalar $1+1$-dimensional wave equations (different from sine-Gordon) we are only aware of the example given in \cite{blank_bruckner_lescarret_schneider:11}. Breathers in discrete nonlinear lattice equations are more common, cf. \cite{mackay_aubry} for a fundamental result and \cite{james_breathers:09} for an overview with many references. Complex valued time-harmonic breathers of the type $u(x,t) = e^{i\omega t} \pmb{u}(x)$ in wave equations with $\mS^1$-equivariant nonlinearity are very well studied objects, cf. \cite{berestycki_lions}, \cite{strauss:1977}. Such time-harmonic breathers are usually much easier to obtain, see Theorem~\ref{ex_complex} below, and they are the object of many papers in the context of the nonlinear Schr\"odinger equation, e.g. in the case of potentials with spatial periodicity \cite{alama_li:1992}, \cite{pankov:2005}.

\medskip

In this paper we consider the $3+1$-dimensional semilinear curl-curl wave equation
\begin{equation}
\usetagform{pm}
\label{semilinear}
s(x) \partial_t^2 U  +\nabla\times\nabla\times U + q(x) U \pm V(x) |U|^{p-1} U = 0 \mbox{ for } (x,t)\in \R^3\times\R
\end{equation}
with $p>1$. We will assume later that $V,q,s:\R^3\to (0,\infty)$ are positive, radially symmetric functions. We look for classical real-valued solutions $U:\R^3\times \R\to \R^3$ which are $T$-periodic in time and spatially exponentially localized, i.e., $\sup_{\R^3\times \R} |U(x,t)| e^{\delta |x|}< \infty$ for some $\delta>0$. We consider both the case of the coefficient $+V(x)$ and $-V(x)$ in front of the nonlinearity. In both cases we have existence results which differ in only one hypothesis. Concerning real-valued breathers of \eqref{semilinear} we are not aware of any other existence result. Our results are as follows. For a function $f: \R^3\to\R$ we say $f(x)\to 0$ in the $C^2$-sense as $x\to 0$ if $f(x), \nabla f(x), D^2 f(x)\to 0$ as $x\to 0$. 

\begin{theorem} Suppose $s,q,V:\R^3\to (0,\infty)$ are radially symmetric $C^2$-functions and let $T=2\pi \sqrt{\frac{s(0)}{q(0)}}$. Assume
\begin{itemize}
\item[(H1)] $T\sqrt{\frac{q(x)}{s(x)}} <  2\pi$ for all $x\in \R^3\setminus\{0\}$,
\item[(H2)] $ \left|2\pi - T\sqrt{\frac{q(x)}{s(x)}}\right|^\frac{1}{p-1} \to 0$ in the $C^2$-sense as $x\to 0$,
\item[(H3)] $\sup_{x\in \R^3} \left|2\pi - T\sqrt{\frac{q(x)}{s(x)}} \right| e^{\delta(p-1)|x|} < \infty$ for some $\delta>0$,
\item[(H4)] $\sup_{x\in\R^3} \frac{q(x)}{V(x)}< \infty$.
\end{itemize}
Then there exists a $T$-periodic $\R^3$-valued breather solution $U$ of \eqref{semilinear}$_+$ with the property $\sup_{\R^3\times\R} |U(x,t)|e^{\delta|x|}<\infty$. The breather $U$ generates a continuum of phase-shifted breathers $U_a(x,t)=U(x,t+a(x))$ where $a:\R^3\to \R$ is an arbitrary radially symmetric $C^2$-function. 
\label{ex_plus}
\end{theorem}

\begin{theorem} Suppose $s,q,V:\R^3\to (0,\infty)$ are radially symmetric $C^2$-functions and let $T=2\pi \sqrt{\frac{s(0)}{q(0)}}$. Assume
\begin{itemize}
\item[(H1)'] $T \sqrt{\frac{q(x)}{s(x)}} >  2\pi$ for all $x\in \R^3\setminus\{0\}$
\end{itemize}
and that (H2)--(H4) hold. Then there exists a $T$-periodic $\R^3$-valued breather solution $U$ of \eqref{semilinear}$_-$ with the property $\sup_{\R^3\times\R} |U(x,t)|e^{\delta|x|}<\infty$. It generates a continuum of phase-shifted breathers $U_a$ as described in Theorem~\ref{ex_plus}.
\label{ex_minus}
\end{theorem}

The search for breather-solutions has spurred a lot of research in the area of nonlinear wave equations. One of the milestones was the discovery of the real-valued breather-family 
$$
u(x,t) = 4 \arctan\left(\frac{m \sin(\omega t)}{\omega \cosh( mx)}\right), \quad m,\omega >0, \quad m^2+\omega^2=1
$$
for the scalar $1+1$-dimensional sine-Gordon equation
\begin{equation}
\partial_t^2 u - \partial_x^2 u + \sin u =0 \mbox{ in } \R\times \R,
\label{sine_gordon}
\end{equation}
cf. \cite{ab_kaup_newell_segur:73}. The nonlinearity $\sin u$ is very special since perturbations of $\sin u$ -- in general -- do not permit breather families, cf. \cite{denzler:93}, \cite{birnir_mckenna_weinstein:94}. The situation is different for scalar $1+1$-dimensional nonlinear wave equations with $x$-dependent coefficients like 
\begin{equation}
\usetagform{pm}
s(x) \partial_t^2 u -\partial_x^2 u + q(x) u \pm V(x) |u|^{p-1} u = 0 \mbox{ for } (x,t)\in \R\times\R.
\label{semi}
\end{equation}
Note that \eqref{semi}$_\pm$ is a special case of \eqref{semilinear}$_\pm$ for fields 
$$
U(x,t) = \begin{pmatrix} 0 \\ 0 \\ u(x_1,t) \end{pmatrix}
$$
since in this case $\divergenz U =0$ and hence
$$
\nabla\times\nabla\times U(x,t)= \begin{pmatrix} 0 \\ 0 \\ -\partial_{x_1}^2 u(x_1,t) \end{pmatrix}.
$$
For \eqref{semi}$_-$, the specific example of $p=3$ and $1-$periodic coefficient functions
\begin{align*}
s(x) &= 1+15\chi_{[6/13,7/14)}(x), \quad x\modd 1 \\
q(x) &= \left(\left(\frac{13\pi}{16}\right)^2 - \left(\frac{13\arccos((9+\sqrt{1881})/100))}{8}\right)^2-\epsilon^2\right)s(x),\\
V(x) &= 1
\end{align*}
given in \cite{blank_bruckner_lescarret_schneider:11} allowed for breather-solutions with minimal period $\frac{32}{13}$ for all $\epsilon\in (0,\epsilon_0]$. This remarkable result relies on tailoring the spectrum of $-y'' =\lambda s(x)y$ and the use of spatial dynamics, center-manifold reduction and bifurcation theory. In a subsequent paper \cite{bruckner_wayne:15} methods of inverse spectral theory were developed that may allow in the future to generalize the above specific example to a bigger class of coefficient functions. 

\medskip

The breather construction of \cite{blank_bruckner_lescarret_schneider:11} strongly exploits the structure of spatially varying coefficients in \eqref{semi}$_-$. Also in our present paper we make heavy use of the particular spatial dependence of the coefficients $s(x), q(x), V(x)$ in \eqref{semilinear}$_\pm$. What is even more important is the particular property of the curl-operator to annihilate gradient fields. This enables us to construct gradient field breathers by ODE-techniques.

\medskip

Let us point out that in our two main theorems we prove the existence of $\R^3$-valued breathers of \eqref{semilinear}. Sometimes, monochromatic complex-valued waves of the type $u(x,t)=e^{i\omega t}\pmb{u}(x)$ are also called breathers provided $\pmb{u}$ decays to $0$ at $\pm\infty$. For such waves the nonlinear hyperbolic problem \eqref{semi}$_\pm$ reduces to a nonlinear ODE problem for $\pmb{u}$:
$$
-\pmb{u}'' +(q(x)-\omega^2 s(x))\pmb{u} \pm V(x) |\pmb{u}|^{p-1} \pmb{u}=0 \mbox{ on } \R
$$
Therefore, as one might expect, many results on the existence of exponentially decaying non-trivial solutions are known, e.g. in the case of periodic potentials \cite{alama_li:1992}, \cite{pankov:2005}. Also for the vector-valued wave equation \eqref{semilinear} one can prove the existence of $\C^3$-valued breathers of the type $e^{\frac{2\pi}{T}it}U(x)$ under various assumptions on the coefficients, cf. \cite{ABDF06}, \cite{bartsch_mederski_2}, \cite{bartsch_dohnal_plum_reichel:14}, \cite{benci_fortunato_archive}, \cite{daprile_siciliano}, \cite{hirsch_reichel}, \cite{Mederski_2014}. However, as the next result shows, it is remarkable that exactly the same assumptions as in Theorem~\ref{ex_plus}, Theorem~\ref{ex_minus} also lead to the existence of $\C^3$-valued monochromatic radially-symmetric breathers. This shows that (H1)--(H4) and (H1)', (H2)--(H4) are in some sense natural assumptions. We are, however, aware of the fact that neither the hypotheses of Theorem~\ref{ex_plus}, Theorem~\ref{ex_minus} nor the hypotheses of Theorem~\ref{ex_complex} are necessarily necessary for the existence of breathers.

%\pagebreak

\begin{theorem} Suppose that $s,q,V:\R^3\to (0,\infty)$ are radially symmetric $C^2$-functions such that (H1)--(H4) or (H1)', (H2)--(H4) hold respectively, and $T=2\pi \sqrt{\frac{s(0)}{q(0)}}$. Then there exists a continuum of $T$-periodic $\C^3$-valued monochromatic breather solutions $U(x,t)=e^{\frac{2\pi}{T} it} \pmb{U}(x)$ of \eqref{semilinear}$_\pm$, respectively, with the property $\sup_{\R^3} |\pmb{U}(x)|e^{\delta|x|}<\infty$.
\label{ex_complex}
\end{theorem} 

Let us finally mention that the interest in breathers in the context of curl-curl nonlinear wave equations stems from the search for optical breathers, i.e., time-periodic spatially localized solutions of Maxwell's equations in anisotropic materials where the permittivity depends nonlinearly on the electromagnetic fields, cf. \cite{Adamashvili_Kaup}. The nonlinear Maxwell problem for the electric field amounts to a quasilinear-in-time curl-curl wave equation, which is much harder to treat than the semilinear problem \eqref{semilinear}. In \cite{pelinovsky_simpson_weinstein} a one-dimensional reduction of the nonlinear Maxwell problem was considered. Based on an approximation by the so-called extended nonlinear coupled mode system (xNLCME) the authors achieved results that indicate the formation and persistence of spatially localized time-periodic polychromatic solutions.

\medskip

The paper is organized as follows. In Chapter~\ref{results} we prove our three main theorems. In order to keep the proofs simple and short we decided to transfer to the Appendix two elementary but slightly lengthy expansions on the inverses of functions given by explicit integral formulas.

\section{Proof of the results} \label{results}

In the above theorems radially symmetric functions from $\R^3\to \R$ occur. For such functions we use the following notation: if $f: \R^3\to \R$ is a radially symmetric $C^2$-function then we denote by $\tilde f: [0,\infty)\to \R$ with $\tilde f(|x|)=f(x)$ its one-dimensional representative, which has the properties $\tilde f\in C^2([0,\infty)$, $\tilde f'(0)=0$. The proofs of the main results require some preparations. We begin with an observation.

\begin{lemma}
Let $\phi: [0, \infty)\to \R$ be a $C^2$-function and let $W: \R^3\setminus\{0\} \to \R^3$ be given by $W(x) := \phi(|x|)\frac{x}{|x|}$. Then $W$ can be extended to a function
\begin{itemize}
\item[(i)] $W \in C^1(\R^3)$ if and only if $\phi(0)=0$.
\item[(ii)] $W \in C^2(\R^3)$ if and only if $\phi(0)=\phi''(0)=0$.
\end{itemize}
\label{observation}
\end{lemma}

\begin{proof} The function $W$ continuously extends to $\R^3$ if and only if $\phi(0)=0$. 

\medskip

(i) The first derivatives of $W$ on $\R^3\setminus\{0\}$ are given by
$$
\frac{\partial W_i}{\partial x_j}(x) = \left(\phi'(r)-\frac{\phi(r)}{r}\right)\frac{x_i x_j}{r^2} + \frac{\phi(r)}{r}\delta_{ij}
$$
where $r=|x|$. For the limit as $x\to 0$ to exist one again needs $\phi(0)=0$. In this case we know that $\phi(r)/r\to \phi'(0)$ as $r\to 0$ and $\phi'(r)-\frac{\phi(r)}{r}\to 0$ as $r\to 0$. This shows that the first partial derivatives of $W$ on $\R^3\setminus\{0\}$ continuously extend to $\R^3$ if and only if $\phi(0)=0$.

\medskip

(ii) The second derivatives of $W$ on $\R^3\setminus\{0\}$ are
$$
\frac{\partial^2 W_i}{\partial x_j \partial x_k}(x) = \underbrace{\left(\phi''(r) - 3 \frac{\phi'(r)r-\phi(r)}{r^2}\right)}_{\to -\frac{1}{2}\phi''(0) \mbox{ as } r\to 0}\frac{x_i x_j x_k}{r^3} + \underbrace{\left(\frac{\phi'(r)r-\phi(r)}{r^2}\right)}_{\to \frac{1}{2}\phi''(0) \mbox{ as } r\to 0}\frac{\delta_{ik}x_j + \delta_{jk}x_i + \delta_{ij}x_k}{r}
$$
Therefore the second partial derivatives of $W$ on $\R^3\setminus\{0\}$ continuously extend to $\R^3$ if and only if $\phi(0)=0$ and $\phi''(0)=0$.
\end{proof}

In the following we consider functions $\psi=\psi(r,t)$ from $[0,\infty)\times \R$ to $\R$. We use the notation $\psi'(r,t)=\frac{\partial}{\partial r} \psi(r,t)$ and $\dot\psi(r,t)= \frac{\partial}{\partial t}\psi(r,t)$. Under the assumptions (H1)--(H4) of Theorem~\ref{ex_plus} or (H1)', (H2)--(H4) of Theorem~\ref{ex_minus} we look for solutions $U$ of \eqref{semilinear}$_\pm$ of the form $U(x,t) := \psi(|x|,t)\frac{x}{|x|}$. 

\begin{lemma}
Let $\psi: [0, \infty)\times \R\to \R$ be a $C^2$-function with $\psi(0,t)=\psi''(0,t)=0$. Then $U(x,t) := \psi(|x|,t)\frac{x}{|x|}$ is a $C^2(\R^3\times\R)$ function. It solves \eqref{semilinear}$_\pm$ if and only if $\psi$ satisfies 
\begin{equation}
\usetagform{pm}
\label{ode_psi}
\tilde s(r)\ddot \psi + \tilde q(r) \psi \pm \tilde V(r)|\psi|^{p-1}\psi = 0 \mbox{ for } r\geq 0, t\in \R.
\end{equation}
\label{ansatz_psi}
\end{lemma}

\begin{proof} By Lemma~\ref{observation} the function $U$ is a $C^2$ function of the variables $x$ and $t$; note that also $\dot\psi(0,t)=\ddot\psi(0,t)=0$. Moreover, by construction $U$ is a gradient-field, i.e., $U(x,t) = \nabla_x \Psi(|x|,t)$ where $\Psi(r,t) = \int_0^r \psi(\rho,t)\,d\rho$. Hence $\nabla\times U=0$ and thus the claim follows.
\end{proof}

The next result is a direct consequence of the fact that \eqref{ode_psi}$_\pm$ is autonomous with respect to $t$ and that $r\geq 0$ plays the role of a parameter.

\begin{lemma} Suppose $U(x,t)=\psi(|x|,t)\frac{x}{|x|}$ solves \eqref{semilinear}$_\pm$. Let $a:\R^3\to \R$ be a radially symmetric $C^2$-function. Then 
$$
U_a(x,t) := U(x,t+a(x))
$$
also solves \eqref{semilinear}$_\pm$. Hence, from one radially symmetric breather one can generate a continuum of different phase-shifted breathers.
\label{phase_shift}
\end{lemma}

The proof of Theorem~\ref{ex_plus} relies on rescaling \eqref{ode_psi}$_+$ as follows: let us find solutions $\psi(r,t)$ of \eqref{ode_psi}$_+$ of the form
\begin{equation}
\psi(r,t) = \tau(r)y(\sigma(r) t).
\label{ansatz}
\end{equation}
Inserting this into \eqref{ode_psi} and comparing coefficients tells us that $y$ has to solve
\begin{equation}
\ddot y + y + |y|^{p-1}y = 0
\label{ode_y_plus}
\end{equation}
where
\begin{equation}
\sigma(r) = \left(\frac{\tilde q(r)}{\tilde s(r)}\right)^{1/2}, \quad \tau(r) = \left(\frac{\tilde q(r)}{\tilde V(r)}\right)^\frac{1}{p-1}.
\label{choice_sigma_tau}
\end{equation}
For the proof of Theorem~\ref{ex_minus} we use the same ansatz \eqref{ansatz} and obtain that $y$ has to solve
\begin{equation}
\ddot y + y - |y|^{p-1}y = 0
\label{ode_y_minus}
\end{equation}
where $\sigma(r)$ and $\tau(r)$ are chosen as in \eqref{choice_sigma_tau}. Next we collect some results about the solutions of \eqref{ode_y_plus} and \eqref{ode_y_minus}.

\begin{lemma}
Define the function $A_+: \R^2\to \R$ by $A_+(\xi,\eta):= \eta^2+\xi^2+ \frac{2}{p+1}|\xi|^{p+1}$. Then $A_+$ is a first integral for \eqref{ode_y_plus}, i.e., every solution $y$ of \eqref{ode_y_plus} satisfies $A_+(y,\dot y)=\const=c$ for some $c\in [0,\infty)$. 
Every orbit of \eqref{ode_y_plus} is uniquely characterized by the value $c\in [0,\infty)$ and every solution $y$ on such an orbit is periodic with minimal period $L(c)$ and maximal amplitude $N(c):= \max_{t\in \R}|y(t)|$. Then 
\begin{itemize}
\item[(i)] $L,N \in C\bigl([0,\infty)\bigr)\cap C^\infty\bigl((0,\infty)\bigr)$ and $L\bigl([0,\infty)\bigr)= (0,2\pi]$, $N\bigl([0,\infty)\bigr)=[0,\infty)$.
\item[(ii)]$N(c)$ is strictly increasing in $c$ with $N'>0$ on $(0,\infty)$, $N(c)\leq \sqrt{c}$ for all $c>0$ and $\lim_{c\to 0} \frac{N(c)}{\sqrt{c}}=1$, $\lim_{c\to \infty} N(c)=\infty$.
\item[(iii)] $L(c)$ is strictly decreasing in $c$ with $L'<0$ on $(0,\infty)$, $\lim_{c\to \infty} L(c)=0$ and $L(0) = 2\pi$.
\item[(iv)] $M= L^{-1}: (0,2\pi]\to [0,\infty)$ is in $C^\infty\bigl((0,2\pi)\bigr)$ and has the following expansions as $s \to 2\pi-$
\begin{align*}
\sqrt{M(s)} & = \sqrt{\alpha}(2\pi-s)^\frac{1}{p-1}(1+O(2\pi-s)), \\
{\sqrt{M(s)}\,}' & = -\frac{\sqrt{\alpha}}{p-1} (2\pi-s)^\frac{2-p}{p-1}(1+ O(2\pi-s)),\\
{\sqrt{M(s)}\,}'' &= \frac{\sqrt{\alpha}(2-p)}{(p-1)^2} (2\pi-s)^\frac{3-2p}{p-1}(1+ O(2\pi-s))
\end{align*}
for some constant $\alpha>0$. 
\end{itemize}
\label{phase_plane_plus}
\end{lemma}

\begin{proof}
Let us first verify all statements for $N(c)$. The function $N(c)$ is given implicitly through
$$
N(c)^2+ \frac{2}{p+1} N(c)^{p+1} = c
$$
which provides the strict monotonicity, continuity and differentiability properties of $N(c)$ for $c>0$. It also implies the inequality $N(c)\leq \sqrt{c}$ and $\lim_{c\to 0} \frac{N(c)}{\sqrt{c}}=1$, $\lim_{c\to\infty} N(c)=\infty$. 

\medskip

Now we prove the statements for $L(c)$. We use the first integral 
$$
|\dot y|^2 + |y|^2 + \frac{2}{p+1}|y|^{p+1} = c
$$
to solve for $\dot y$ in all four quadrants of the phase-plane, cf. Figure~\ref{figure_plus}.
\begin{figure}
\scalebox{0.5}{\includegraphics{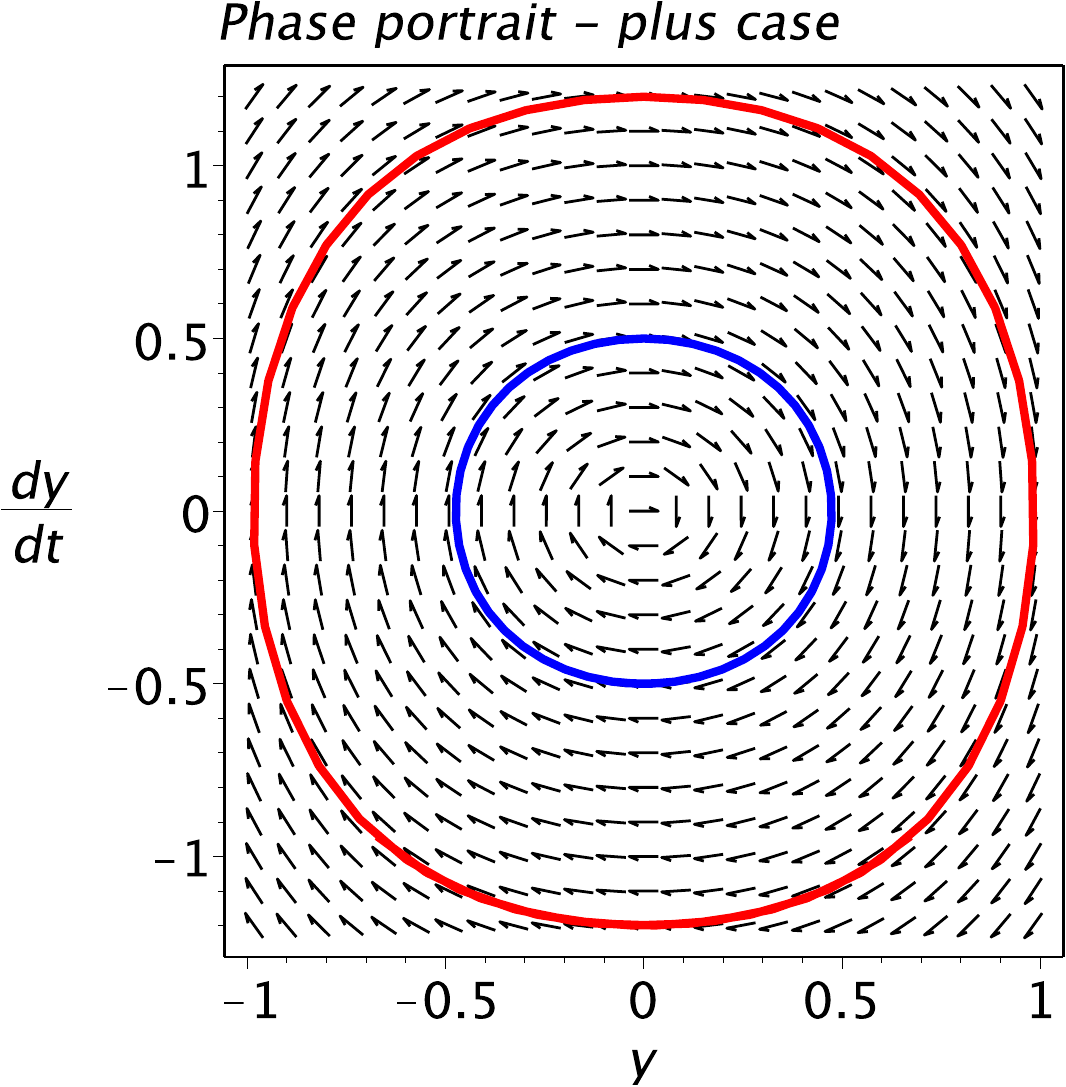}}
\caption{Part of the phase plane of \eqref{ode_y_plus} for $p=3$ with two periodic orbits.}
\label{figure_plus}
\end{figure}
Together with the defining equation for $N(c)$ this yields 
\begin{align}
L(c) & = 4 \int_0^{N(c)} \frac{1}{\sqrt{c-y^2-\frac{2}{p+1}y^{p+1}}}\,dy  \nonumber\\
& = 4 \int_0^1 \frac{N(c)}{\sqrt{c-N(c)^2z^2- \frac{2}{p+1} N(c)^{p+1}z^{p+1}}}\,dz  \label{L_darstellung_plus}\\
& = 4 \int_0^1 \frac{1}{\sqrt{1-z^2+\frac{2}{p+1} N(c)^{p-1}(1-z^{p+1})}} \,dz. \nonumber
\end{align}
Now we find that $L(c)$ has the asserted smoothness properties and is strictly decreasing with $L'<0$ on $(0,\infty)$, $\lim_{c\to \infty} L(c)=0$ and 
$$
\lim_{c\to 0} L(c) = 4 \int_0^1 \frac{1}{\sqrt{1-z^2}}\,dz = 2\pi.
$$
The fact that $L'<0$ on $(0,\infty)$ follows from  $N'>0$ on $(0,\infty)$ and \eqref{f_prime_plus} in the proof of Lemma~\ref{expand_plus}. This yields also that $M=L^{-1}\in C^\infty\bigl((0,2\pi)\bigr)$. The expansions for $\sqrt{M}$ and its derivatives  can be found in Lemma~\ref{expand_plus} in the Appendix.
\end{proof}

\begin{lemma}
Define the function $A_-: \R^2\to \R$ by $A_-(\xi,\eta):= \eta^2+\xi^2- \frac{2}{p+1}|\xi|^{p+1}$. Then $A_-$ is a first integral for \eqref{ode_y_minus}, i.e., every solution $y$ of \eqref{ode_y_minus} satisfies $A_-(y,\dot y)=\const=c$ for some $c\in \R$. Every bounded orbit of \eqref{ode_y_minus} is uniquely characterized by the value $c\in [0,\frac{p-1}{p+1}]$ and for $c\in [0,\frac{p-1}{p+1})$ every solution $y$ on such an orbit is periodic with minimal period $L(c)$ and maximal amplitude $N(c):= \max_{t\in \R}|y(t)|$. Then 
\begin{itemize}
\item[(i)] $L,N \in C\bigl([0, \frac{p-1}{p+1})\bigr)\cap C^\infty \bigl((0,\frac{p-1}{p+1})\bigr)$ and $L\bigl([0,\frac{p-1}{p+1})\bigr)=[2\pi,\infty)$, $N\bigl([0,\frac{p-1}{p+1})\bigr)=[0,1)$.
\item[(ii)]$N$ is strictly increasing in $c$ with $N'>0$ on $(0,\frac{p-1}{p+1})$ , $N(c)\leq \sqrt{\frac{p+1}{p-1}c}$ for all $c\in [0,\frac{p-1}{p+1})$ and $\lim_{c\to \frac{p-1}{p+1}} N(c)=1$, $\lim_{c\to 0} \frac{N(c)}{\sqrt{c}}=1$.
\item[(iii)] $L(c)$ is strictly increasing in $c$ with $L'>0$ on $(0,\frac{p-1}{p+1})$, $\lim_{c\to \frac{p-1}{p+1}} L(c)=\infty$ and $L(0) = 2\pi$.
\item[(iv)] $M= L^{-1}: [2\pi,\infty) \to [0,\frac{p-1}{p+1})$ is in $C^\infty\bigl((2\pi,\infty)\bigr)$ and has the following expansions as $s \to 2\pi+$
\begin{align*}
\sqrt{M(s)} &= \sqrt{\alpha}(s-2\pi)^\frac{1}{p-1}(1+O(s-2\pi)), \\
{\sqrt{M(s)}\,}' &= \frac{\sqrt{\alpha}}{p-1} (s-2\pi)^\frac{2-p}{p-1}(1+ O(s-2\pi)), \\
{\sqrt{M(s)}\,}'' &= \frac{\sqrt{\alpha}(2-p)}{(p-1)^2} (s-2\pi)^\frac{3-2p}{p-1}(1+ O(s-2\pi))
\end{align*}
for the same constant $\alpha>0$ as in Lemma~\ref{phase_plane_plus}. 
\end{itemize}
\label{phase_plane_minus}
\end{lemma}

\begin{proof} Besides the equilibrium $(0,0)$ there are two further equilibria $(\pm 1,0)$ connected by two heteroclinic orbits. The first integral 
$$
|\dot y|^2 + |y|^2 - \frac{2}{p+1}|y|^{p+1} = c
$$
leads to closed orbits for $0< c < A_-(\pm 1,0)=\frac{p-1}{p+1}$ and provided initial conditions are chosen in the bounded component of the set $A_-^{-1}\bigl([0,\frac{p-1}{p+1})\bigr)$, cf. Figure~\ref{figure_minus}.
\begin{figure}
\scalebox{0.5}{\includegraphics{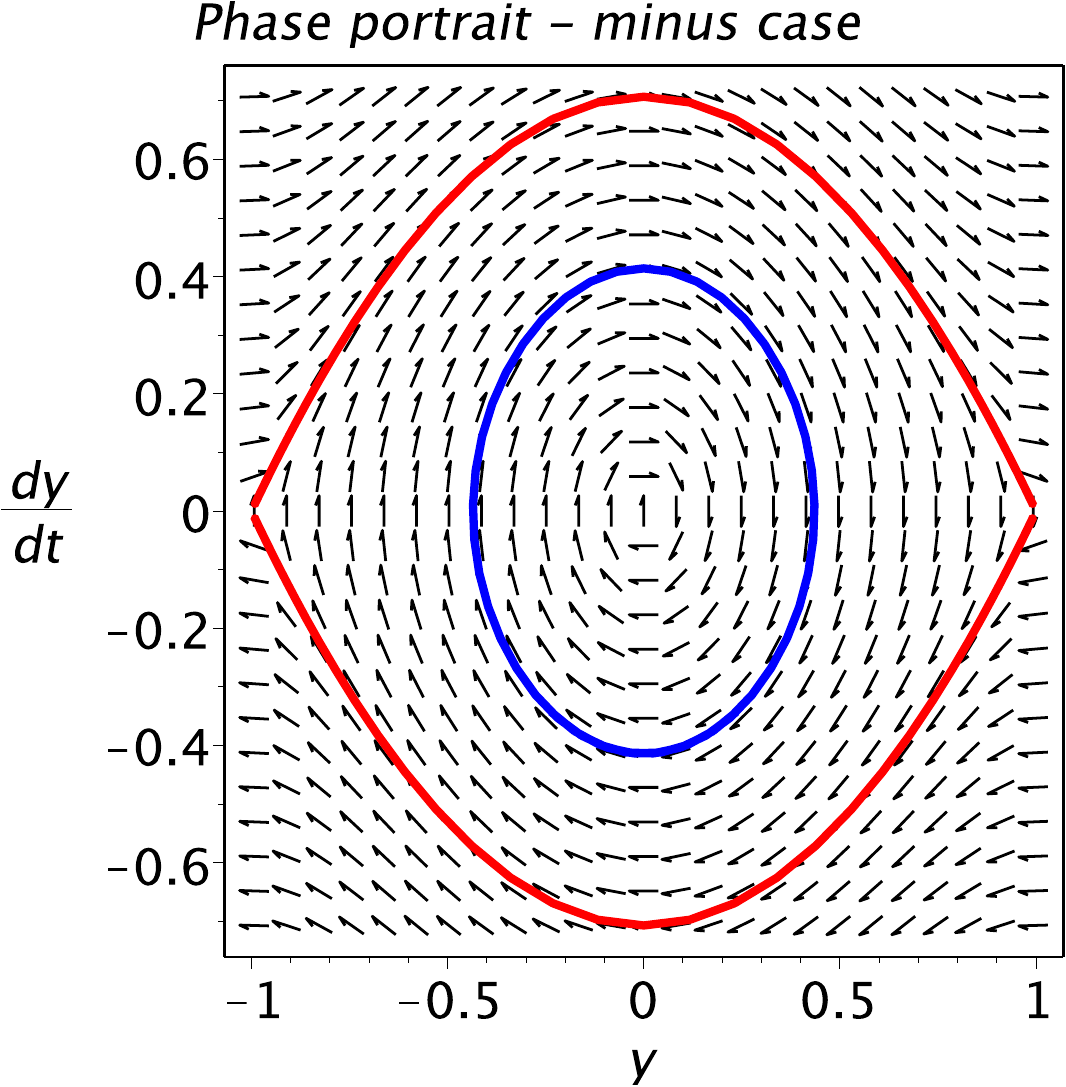}}
\caption{Part of the phase plane of \eqref{ode_y_minus} for $p=3$ with a periodic orbit (blue) and two heteroclinic connections (red).}
\label{figure_minus}
\end{figure}
The defining equation for the function $N(c)$ is
\begin{equation}
N(c)^2- \frac{2}{p+1} N(c)^{p+1} = c \mbox{ and } N(c) < 1
\label{def_N_minus}
\end{equation}
which provides the strict monotonicity, continuity and differentiability properties of $N(c)$ for $0<c<\frac{p-1}{p+1}$. Since $N(c)<1$ we obtain from \eqref{def_N_minus} the inequality $N(c)\leq \sqrt{\frac{p+1}{p-1} c}$, the fact that $N'>0$ on $(0,\frac{p-1}{p+1})$, and $\lim_{c\to 0} \frac{N(c)}{\sqrt{c}}=1$. Moreover, $N(c)\to 1$ as $c\nearrow \frac{p-1}{p+1}$. This completes the statements on $N(c)$. 

\medskip

Now we turn to $L(c)$. This time \eqref{def_N_minus} together with the first integral yields 
\begin{align}
L(c) & = 4 \int_0^{N(c)} \frac{1}{\sqrt{c-y^2+\frac{2}{p+1}y^{p+1}}}\,dy  \nonumber\\
& = 4 \int_0^1 \frac{N(c)}{\sqrt{c-N(c)^2z^2+\frac{2}{p+1} N(c)^{p+1}z^{p+1}}}\,dz  \label{L_darstellung_minus}\\
& = 4 \int_0^1 \frac{1}{\sqrt{1-z^2-\frac{2}{p+1} N(c)^{p-1}(1-z^{p+1})}} \,dz. \nonumber
\end{align} 
Clearly $L(0)=2\pi$. Moreover $L(c)$ has the asserted smoothness properties, and is strictly increasing since $N(c)$ is strictly increasing for $c\in [0,\frac{p-1}{p+1}]$. Since $\lim_{c\to \frac{p-1}{p+1}} N(c)=1$ and $1-z^2-\frac{2}{p+1}(1-z^{p+1})= (p-1)(1-z)^2(1+o(1))$ as $z\to 1$ we see now that $\lim_{c\to \frac{p-1}{p+1}} L(c)=\infty$. The fact that $L'>0$ on $(0,\frac{p-1}{p+1})$ follows from  $N'>0$ on $(0,\infty)$ and \eqref{f_prime_minus} in the proof of Lemma~\ref{expand_plus}. This yields also that $M=L^{-1}\in C^\infty\bigl((2\pi,\infty)\bigr)$. The expansions for $\sqrt{M}$ and its derivatives  can be found in Lemma~\ref{expand_minus} in the Appendix.
\end{proof}

\noindent
{\bf Proof of Theorem~\ref{ex_plus}:} We begin by choosing a $C^2$-curve $\gamma: [0,\infty)\to \R^2$ in phase space such that $A_+(\gamma(c))=c^2$, where $A_+$ is the first integral from Lemma~\ref{phase_plane_plus}. Such a curve is e.g. given by $\gamma(c)=(0,c)$. There is a continuum of other possible choices of $\gamma$. The choice of $\gamma$ actually only selects a particular member of the continuum of phase-shifted breathers as described in Lemma~\ref{phase_shift} (we will comment on this aspect at the end of the proof).

\medskip

Let us denote by $y(t;c)$ the solution of \eqref{ode_y_plus} with $\bigl(y(0;c),\dot y(0;c))\bigr)=\gamma(c)$. Then $y:\R\times [0,\infty)\to \R$ is a $C^2$-function and $y(t;c)$ is $L(c^2)$-periodic in the $t$-variable. Now we define the solution $\psi$ of \eqref{ode_psi}$_+$ by 
\begin{equation}
\psi(r,t) := \tau(r) y(\sigma(r)t; c) \quad \mbox{ with } \quad \sigma(r) = \left(\frac{\tilde q(r)}{\tilde s(r)}\right)^{1/2}, \quad \tau(r) = \left(\frac{\tilde q(r)}{\tilde V(r)}\right)^\frac{1}{p-1},
\label{def_psi_plus}
\end{equation}
see \eqref{ansatz}, \eqref{choice_sigma_tau}. The requirement of $T$-periodicity of $\psi$ in the $t$-variable tells us how to choose $c$ as a function of the radial variable $r\in [0,\infty)$, i.e.,
$$
g(r) := \sigma(r) T \stackrel{!}{=} L(c^2). 
$$
Recall from Lemma~\ref{phase_plane_plus} the definition $M=L^{-1}$ and that $M:(0,2\pi]\to \R$ is strictly decreasing and $C^\infty$ on $(0,2\pi)$. Now
\begin{equation}
\label{def_c_plus}
c(r) = \sqrt{M(g(r))}
\end{equation}
has to be inserted into \eqref{def_psi_plus}. Note that the assumption (H1) of Theorem~\ref{ex_plus} guarantees that $c(r)$ is well-defined and $C^2$ on $(0,\infty)$. Next we show that $\psi(r,t)$ tends to $0$ as $r\to 0$ and is even exponentially decaying to $0$ as $r\to \infty$. First note the estimate 
\begin{align*}
|\psi(r,t)| & \leq \left(\frac{\tilde q(r)}{\tilde V(r)}\right)^\frac{1}{p-1} N(c(r)^2)  \\
 & \leq \underbrace{\left(\frac{\tilde q(r)}{\tilde V(r)}\right)^\frac{1}{p-1}}_{\leq B}c(r) \mbox{ by assumption (H4) and Lemma~\ref{phase_plane_plus}(ii)}\\
 & \leq B \sqrt{M(g(r))}. 
\end{align*}
By assumptions (H2) and (H3) of Theorem~\ref{ex_plus} the argument of $M$ in the above inequality tends to $2\pi$ as $r\to \infty$ and as $r\to 0$. By Lemma~\ref{phase_plane_plus}(iv) we have the estimate 
\begin{equation}
\label{est_psi_plus}
|\psi(r,t)| \leq  B \sqrt{\alpha}\left(2\pi-g(r)\right)^\frac{1}{p-1}O(1)  \mbox{ as } r\to \infty \mbox{ and as } r\to 0.
\end{equation}
Assumption (H3) of Theorem~\ref{ex_plus} and \eqref{est_psi_plus} yield 
$|\psi(r,t)| \leq  C \exp(-\delta r)$ for $r\geq 0$ which proves the exponential decay of $U(x,t)=\psi(|x|,t)\frac{x}{|x|}$ as $|x|\to\infty$. 

\medskip

Next we see that \eqref{est_psi_plus} and (H2) imply $\psi(0,t)=0$. In order to apply Lemma~\ref{ansatz_psi} it remains to prove $\psi \in C^2([0,\infty)\times \R)$ and that $\psi''(0,t)=0$. For this we compute from \eqref{def_c_plus} that $c(r)=\sqrt{\alpha}(2\pi-g(r))^\frac{1}{p-1}O(1)\to 0$ as $r\to 0$. Furthermore \eqref{def_c_plus} implies
\begin{align*}
c'(r) &= {\sqrt{M}\,}'(g(r))g'(r) \\
&=-\frac{\sqrt{\alpha}}{p-1} (2\pi-g(r))^\frac{2-p}{p-1}O(1)g'(r)  \quad \mbox{ by Lemma~\ref{phase_plane_plus}(iv)}\\
&= \sqrt{\alpha} \left((2\pi-g(r))^\frac{1}{p-1}\right)'O(1)\\
&= o(1) \mbox{ as $r\to 0$ by assumption (H2).}
\end{align*}
Likewise 
\begin{align*}
c''(r) =& {\sqrt{M}\,}''(g(r))g'(r)^2 + {\sqrt{M}\,}'(g(r))g''(r) \\
=& \frac{\sqrt{\alpha}(2-p)}{(p-1)^2} (2\pi-g(r))^\frac{3-2p}{p-1}\bigl(1+O(2\pi-g(r))\bigr)g'(r)^2 \\
& -\frac{\sqrt{\alpha}}{p-1} (2\pi-g(r))^\frac{2-p}{p-1}\bigl(1+O(2\pi-g(r))\bigr)g''(r) \\
= & \sqrt{\alpha}\underbrace{\left((2\pi-g(r))^\frac{1}{p-1}\right)''}_{=:T_1} +O(1) \frac{\sqrt{\alpha}(2-p)}{(p-1)^2}\underbrace{(2\pi-g(r))^\frac{2-p}{p-1}g'(r)^2}_{=:T_2} \\
& - O(1) \frac{\sqrt{\alpha}}{p-1}\underbrace{(2\pi-g(r))^\frac{1}{p-1}g''(r)}_{=:T_3}. 
\end{align*}
The term $T_1$ converges to $0$ as $r\to 0$ by assumption (H2). Recall that $g$ is a $C^2$-function on $[0,\infty)$. The term $T_3$ converges to $0$ since $g(r)\to 2\pi$ as $r\to 0$ and $g''$ is bounded near $0$. And since $T_2(r)=(1-p)\left((2\pi-g(r))^\frac{1}{p-1}\right)'g'(r)$ with $g'$ being bounded near $0$ we see that (H2) also implies $T_2(r)\to 0$ as $r\to 0$. This shows that $c''(r)\to 0$ as $r\to 0$. Hence $c$ can be extended to a function $c\in C^2\bigl([0,\infty)\bigr)$ with $c(0)=c'(0)=c''(0)=0$. Having this and recalling $\psi(r,t)=\tau(r)y(\sigma(r)t,c(r))$ we see that $\psi \in C^2([0,\infty)\times \R)$. Hence we may compute
\begin{align*}
\psi'(r,t) & = \tau'(r) y(\sigma(r)t, c(r)) + \tau(r) \dot y(\sigma(r)t, c(r))\sigma'(r) t + \tau(r) \frac{\partial y}{\partial c}(\sigma(r)t,c(r))c'(r)
\end{align*}
and 
\begin{align*}
\psi''(0,t) =& \tau''(0) \underbrace{y(\sigma(0)t,c(0))}_{=0} + 2 \tau'(0)\underbrace{\dot y(\sigma(0)t,c(0))}_{=0}\sigma'(0)t +2 \tau'(0)\frac{\partial y}{\partial c}(\sigma(0)t,c(0))\underbrace{c'(0)}_{=0} \\
& + \tau(0)\underbrace{\ddot y(\sigma(0)t,c(0))}_{=0}\sigma'(0)^2t^2+ \tau(0)\underbrace{\dot y(\sigma(0)t,c(0))}_{=0}\sigma''(0)t \\
& + 2\tau(0)\frac{\partial \dot y}{\partial c}(\sigma(0)t,c(0))\sigma'(0)t\underbrace{c'(0)}_{=0} +\tau(0)\frac{\partial^2 y}{\partial c^2}(\sigma(0)t,c(0))\underbrace{c'(0)^2}_{=0} \\
& + \tau(0)\frac{\partial y}{\partial c}(\sigma(0)t,c(0))\underbrace{c''(0)}_{=0} \\
=&0,
\end{align*}
where we have used $y(\cdot,0)=0$, $\dot y(\cdot,0)=0$, $\ddot y(\cdot,0)=0$. By Lemma~\ref{ansatz_psi} this implies that $U\in C^2(\R^3\times \R)$. The asserted continuum of solutions is now given by Lemma~\ref{phase_shift}. This finishes the proof of Theorem~\ref{ex_plus}.  

\medskip

Now we will comment on the choice of the initial curve $\gamma(c)=(0,c)$ which led to the solution family $y(t;c)$ such that $(y(0;c), \dot y(0;c))=\gamma(c)$. Our objective was to determine \emph{some} $C^2$-curve such that $A_+(\gamma(c))=c^2$. The particular choice $\gamma(c)=(0,c)$ is convenient but arbitrary. Let us explain other possible choices of $\gamma$. E.g. take
$$
\tilde \gamma(c) := \bigl(y(b(c);c), \dot y(b(c);c)\bigr)
$$
for an arbitrary function $b\in C^2([0,\infty);\R)$. Clearly, $A_+(\tilde\gamma(c))=A_+\left(y(b(c);c), \dot y(b(c);c)\right)=c^2$ since $A_+$ is a first integral of \eqref{ode_y_plus}. With the new curve $\tilde \gamma$ we can define a new solution family $\tilde y(t;c)$ through the initial conditions
$$
\bigl(\tilde y(0;c), \dot{\tilde y}(0;c)\bigr) = \tilde\gamma(c)
$$
By uniqueness of the initial value problem the new and old solution families have the simple relation
$$
\tilde y(t;c) = y(t+b(c);c).
$$
In order to see the effect of the choice of the new curve let us compare the solutions $U$, $\tilde U$ generated by $\gamma$, $\tilde\gamma$, i.e.,
$$
U(x,t) = \tau(r) y(\sigma(r) t;c(r))\frac{x}{|x|},
$$
where $c(r)=\sqrt{L^{-1}(\sigma(r))}$. Likewise
\begin{align*}
\tilde U(x,t) &= \tau(r) \tilde y(\sigma(r)t; c(r))\frac{x}{|x|} \\
&= \tau(r)y(\sigma(r)t+b(c(r));c(r))\frac{x}{|x|} \\
&= U(x,t+a(r)),
\end{align*}
where $a(r) = b(c(r))/\sigma(r)$ is a $C^2$-function on $[0,\infty)$. Hence, this different choice of the initial curve led to a phase-shifted breather as already explained in Lemma~\ref{phase_shift}.
\qed

\medskip

\noindent
{\bf Proof of Theorem~\ref{ex_minus}:} Again we choose a $C^2$-curve $\gamma: [0,\frac{p-1}{p+1})\to \R^2$ in phase space such that $A_-(\gamma(c))=c$. Now $A_-$ is the first integral from Lemma~\ref{phase_plane_minus}. As before, such a curve is e.g. $\gamma(c)=(0,c)$. The fact that other choices of $\gamma$ are also possible and just lead to a phase shift as shown in Lemma~\ref{phase_shift} has already been explained at the end of the proof of Theorem~\ref{ex_plus}. We denote by $y(t;c)$ the solution of \eqref{ode_y_minus} with $\bigl(y(0;c),\dot y(0;c))\bigr)=\gamma(c)$. Then $y:\R\times [0,\frac{p-1}{p+1})\to \R$ is a $C^2$-function and $y(t;c)$ is $L(c)$-periodic in the $t$-variable. A solution $\psi$ of \eqref{ode_psi}$_-$ is then defined by 
\begin{equation}
\psi(r,t) := \tau(r) y(\sigma(r)t; c) \quad \mbox{ with } \quad \sigma(r), \tau(r) \mbox{ as previously.}
\label{def_psi_minus}
\end{equation}
The condition of $T$-periodicity of $\psi$ in the $t$-variable is the same as before and requires 
$$
g(r) := \sigma(r) T \stackrel{!}{=} L(c^2). 
$$
Now the inverse $M=L^{-1}$ is defined on $[2\pi,\infty)\to \R$ as a continuous, strictly increasing function which is $C^\infty$ on $(2\pi,\infty)$, cf. Lemma~\ref{phase_plane_minus}. Assumption (H1)' of Theorem~\ref{ex_minus} guarantees that
$$
c(r) = \sqrt{M(g(r))}
$$
is well-defined and $C^2$ on $(0,\infty)$. Inserting $c(r)$  into \eqref{def_psi_minus} yields a $T$-periodic solution $\psi(r,t)$ of \eqref{ode_psi}$_-$. We proceed via the estimate
\begin{align*}
|\psi(r,t)| & \leq \left(\frac{\tilde q(r)}{\tilde V(r)}\right)^\frac{1}{p-1} N(c(r)^2)  \\
 & \leq \underbrace{\left(\frac{\tilde q(r)}{\tilde V(r)}\right)^\frac{1}{p-1}}_{\leq B} \sqrt{\frac{p+1}{p-1}}c(r) \mbox{ by assumption (H4) and Lemma~\ref{phase_plane_minus}(ii)}\\
 & = B \sqrt{\frac{p+1}{p-1}}\sqrt{M(g(r)}.
\end{align*}
As before, (H2) and (H3) imply that the argument of $M$ in the above inequality tends to $2\pi$ as $r\to \infty$ and as $r\to 0$. Making use of the estimate in Lemma~\ref{phase_plane_minus}(iv) we obtain
\begin{equation}
\label{est_psi_minus}
|\psi(r,t)| \leq  B\sqrt{\frac{p+1}{p-1}} \sqrt{\alpha}\left(g(r)-2\pi\right)^\frac{1}{p-1}O(1)  \mbox{ as } r\to \infty \mbox{ and as } r\to 0.
\end{equation}
As before assumption (H3) leads to the exponential decay of $U(x,t)$ as $|x|\to \infty$. Similarly to the proof of Theorem~\ref{ex_plus} the expansions of $\sqrt{M}$, ${\sqrt{M}\,}'$ and ${\sqrt{M}\,}''$ and (H2) imply $c'(0)=c''(0)=0$ which leads in an identical way as before to $\psi''(0,t)=0$ and thus $U\in C^2(\R^3\times\R)$.
\qed

\medskip

\noindent
{\bf Proof of Theorem~\ref{ex_complex}:} We use the ansatz $U(x,t)= \phi(|x|) e^{i\frac{2\pi}{T}t} \frac{x}{|x|}$. According to Lemma~\ref{ansatz_psi} it represents a $T$-periodic breather if $\phi: [0,\infty)\to \R$ is a $C^2$-solution of 
$$
-\left(\frac{2\pi}{T}\right)^2 \tilde s(r) + \tilde q(r) \pm \tilde V(r) |\phi(r)|^{p-1} = 0 \mbox{ with } \phi(0)=\phi''(0)=0
$$
which exponentially decays to zero at $\infty$. This can be satisfied for 
$$
\phi(r) := \left[\pm\left(\left(\frac{2\pi}{T}\right)^2 \frac{\tilde s(r)}{\tilde q(r)}-1\right)\frac{\tilde q(r)}{\tilde V(r)}\right]^\frac{1}{p-1}. 
$$
The assumptions (H1), (H1)' guarantee that $\phi$ is well-defined. By (H3), (H4) it is exponentially decreasing as $r\to \infty$ and by (H2) we see that $\phi(0)=\phi''(0)=0$ so that $U(x,t)$ is a classical solution of \eqref{semilinear}$_\pm$ on $\R^3\times \R$. \qed

\section*{Appendix}

\begin{lemma}[Expansion of $M=L^{-1}$ for \eqref{ode_y_plus}] \label{expand_plus}
$M: (0,2\pi] \to [0,\infty)$ is $C^\infty$ on $(0,2\pi)$ and has the following expansions as $s \to 2\pi-$
\begin{align*}
M(s) & = \alpha (2\pi-s)^\frac{2}{p-1}(1+O(2\pi-s)), \\
\sqrt{M(s)} & = \sqrt{\alpha}(2\pi-s)^\frac{1}{p-1}(1+O(2\pi-s)), \\
M'(s) &= -\frac{2\alpha}{p-1}(2\pi-s)^\frac{3-p}{p-1}(1+ O(2\pi-s)), \\
{\sqrt{M(s)}\,}' & = -\frac{\sqrt{\alpha}}{p-1} (2\pi-s)^\frac{2-p}{p-1}(1+ O(2\pi-s)),\\
M''(s) &= \frac{2\alpha(3-p)}{(p-1)^2} (2\pi-s) ^\frac{4-2p}{p-1}(1+O(2\pi-s)), \\
{\sqrt{M(s)}\,}'' &= \frac{\sqrt{\alpha}(2-p)}{(p-1)^2} (2\pi-s)^\frac{3-2p}{p-1}(1+ O(2\pi-s))
\end{align*}
for some constant $\alpha>0$.
\end{lemma}

\begin{proof}
Let us begin by recalling from \eqref{L_darstellung_plus} that
$$
L(c) = F\left(\frac{2}{p+1} N(c)^{p-1}\right), 
$$
where 
\begin{align*}
F(w)&= 4 \int_0^1 \frac{1}{\sqrt{1-z^2+ w(1-z^{p+1})}} \,dz \\
&= 4\int_0^1 \frac{1}{\sqrt{1-z^2}\sqrt{1 +w \kappa(z)}}\,dz\quad\mbox{ with } \kappa(z)=\frac{1-z^{p+1}}{1-z^2} \mbox{ and } w\geq 0.
\end{align*}
Since $\kappa$ is a continuous and positive function on $[0,1]$ we find that $F\in C^\infty[0,\infty)$, $F(0)=2\pi$, $F(\infty)=0$ and $F$ is strictly decreasing and convex with 
\begin{align}
F'(w)&= -2\int_0^1 \frac{\kappa(z)}{\sqrt{1-z^2}(1 +w \kappa(z))^\frac{3}{2}}\,dz<0,  \label{f_prime_plus} \\
F''(w) &= 3\int_0^1 \frac{\kappa^2(z)}{\sqrt{1-z^2}(1 +w \kappa(z))^\frac{5}{2}} \,dz>0 \mbox{ for } w\in [0,\infty).
\end{align}
Thus $F^{-1} \in C^\infty((0,2\pi])$.

\medskip

Our objective is to study $L^{-1}$. Recall from the defining equation for $N(c)$ that for $w=\frac{2}{p+1} N(c)^{p-1}$ one has the relation
\begin{align*}
c = N(c)^2+ \frac{2}{p+1} N(c)^{p+1} &= \left(\frac{p+1}{2}\right)^\frac{2}{p-1} \left(w^\frac{2}{p-1} + w^\frac{p+1}{p-1}\right) \\
&=: \Phi(w).
\end{align*}
This leads to the representation
$$
L(c) = F(\Phi^{-1}(c)) \quad \mbox{ and } \quad 
M = L^{-1} = \Phi\circ F^{-1}.
$$
Via Taylor-approximation with $\tilde\alpha=-1/F'(0)>0$, $\tilde\beta= F''(0)>0$ we obtain as $s\to 2\pi-$
\begin{align*}
F^{-1}(s) &= (F^{-1})'(2\pi)(s-2\pi) + O((2\pi-s)^2) = \tilde\alpha(2\pi-s)(1+O(2\pi-s)), \\
(F^{-1})'(s) &= \frac{1}{F'(F^{-1}(s))} = -\tilde\alpha(1+O(2\pi-s)), \\
(F^{-1})''(s) &= -\frac{F''\bigl(F^{-1}(s)\bigr)}{\bigl(F'(F^{-1}(s))\bigr)^3} = \tilde\beta\tilde\alpha^3(1+O(2\pi-s)).
\end{align*}
Hence as $s\to 2\pi-$ we obtain
$$
M(s) = \Phi(F^{-1}(s)) = \alpha(2\pi-s)^\frac{2}{p-1}(1+O(2\pi-s))
$$
for $\alpha=  \left(\frac{p+1}{2}\right)^\frac{2}{p-1} \tilde\alpha^\frac{2}{p-1} >0$, 
\begin{align*}
M'(s) & = \Phi'(F^{-1}(s)) (F^{-1})'(s) \\
& = \left(\frac{p+1}{2}\right)^\frac{2}{p-1}\left( \frac{2}{p-1} (F^{-1}(s))^\frac{3-p}{p-1}+\frac{p+1}{p-1}(F^{-1}(s))^\frac{2}{p-1}\right)(-\tilde\alpha)(1+O(2\pi-s)) \\
& = \frac{-2\alpha}{p-1}(2\pi-s)^\frac{3-p}{p-1}(1+O(2\pi-s)) 
\end{align*}
and 
\begin{align*}
M''(s) =&  \Phi''(F^{-1}(s)) \Bigl((F^{-1})'(s)\Bigr)^2 + \Phi'(F^{-1}(s)) (F^{-1})''(s) \\
=& \left(\frac{p+1}{2}\right)^\frac{2}{p-1}\left( \frac{2(3-p)}{(p-1)^2} (F^{-1}(s))^\frac{4-2p}{p-1} + \frac{2(p+1)}{(p-1)^2}(F^{-1}(s))^\frac{3-p}{p-1}\right)\tilde\alpha^2(1+O(2\pi-s))\\
& + O\bigl((2\pi-s)^\frac{3-p}{p-1}\bigr) \\
=& \frac{2\alpha(3-p)}{(p-1)^2}(2\pi-s)^\frac{4-2p}{p-1}(1+O(2\pi-s)).
\end{align*}
The expansions for $\sqrt{M}, {\sqrt{M}\,}' = \frac{M'}{2\sqrt{M}}$ and ${\sqrt{M}\,}'' = \frac{1}{2 M^{3/2}}\left(M''M-\frac{1}{2} (M')^2\right)$ follow directly from the expansions for $M, M', M''$. 
\end{proof}

\begin{lemma}[Expansion of $M=L^{-1}$ for \eqref{ode_y_minus}] \label{expand_minus}
$M: [2\pi,\infty) \to [0,\frac{p-1}{p+1})$ is $C^\infty$ on $(2\pi,\infty)$ and has the following expansions as $s \to 2\pi+$
\begin{align*}
M(s) & = \alpha (s-2\pi)^\frac{2}{p-1}(1+O(s-2\pi)), \\
\sqrt{M(s)} &= \sqrt{\alpha}(s-2\pi)^\frac{1}{p-1}(1+O(s-2\pi)), \\
M'(s) &= \frac{2\alpha}{p-1}(s-2\pi)^\frac{3-p}{p-1}(1+ O(s-2\pi)), \\
{\sqrt{M(s)}\,}' &= \frac{\sqrt{\alpha}}{p-1} (s-2\pi)^\frac{2-p}{p-1}(1+ O(s-2\pi)), \\
M''(s) &= \frac{2\alpha(3-p)}{(p-1)^2} (s-2\pi) ^\frac{4-2p}{p-1}(1+O(s-2\pi)),\\
{\sqrt{M(s)}\,}'' &= \frac{\sqrt{\alpha}(2-p)}{(p-1)^2} (s-2\pi)^\frac{3-2p}{p-1}(1+ O(s-2\pi))
\end{align*}
with the same constant $\alpha>0$ as in Lemma~\ref{expand_plus}.
\end{lemma}
\begin{proof}
Let us begin by recalling from \eqref{L_darstellung_minus} that
$$
L(c) = F\left(\frac{2}{p+1} N(c)^{p-1}\right), 
$$
where 
\begin{align*}
F(w) &= 4 \int_0^1 \frac{1}{\sqrt{1-z^2-w(1-z^{p+1})}} \,dz \\
&= 4\int_0^1 \frac{1}{\sqrt{1-z^2}\sqrt{1 -w \kappa(z)}}\,dz\quad\mbox{ with } \kappa(z)=\frac{1-z^{p+1}}{1-z^2}.
\end{align*}
Since $\kappa$ is a continuous function on $[0,1]$ which takes values only in $[1,\frac{p+1}{2}]$ we find that $F(w)$ is well defined for $w\in [0,\frac{2}{p+1})$, $F\in C^\infty[0,\frac{2}{p+1})$ and $F(0)=2\pi$. Moreover, the Taylor-expansion of $\kappa$ at $1$ yields 
$$
\kappa(z) = \frac{p+1}{2} + \frac{p^2-1}{4}(z-1)(1+o(1)) \mbox{ as } z \to 1-
$$
so that $\lim_{w\to \frac{2}{p+1}} F(w)=\infty$. Finally, $F$ is strictly increasing and convex with 
\begin{align}
F'(w)&= 2\int_0^1 \frac{\kappa(z)}{\sqrt{1-z^2}(1 -w \kappa(z))^\frac{3}{2}}\,dz>0, \label{f_prime_minus} \\
F''(w) &= 3\int_0^1 \frac{\kappa^2(z)}{\sqrt{1-z^2}(1 -w \kappa(z))^\frac{5}{2}} \,dz>0
\mbox{ for } w\in [0,\frac{2}{p+1})
\end{align}
and hence $F^{-1} \in C^\infty\bigl([2\pi,\infty)\bigr)$.  

\medskip

Our objective is to study $L^{-1}$. Recall from the defining equation for $N(c)$ that for $w=\frac{2}{p+1} N(c)^{p-1}$ one has the relation
\begin{align*}
c = N(c)^2- \frac{2}{p+1} N(c)^{p+1} &= \left(\frac{p+1}{2}\right)^\frac{2}{p-1} \left(w^\frac{2}{p-1} - w^\frac{p+1}{p-1}\right) \\
&=: \Phi(w).
\end{align*}
This leads to the representation
$$
L(c) = F(\Phi^{-1}(c)) \quad \mbox{ and } \quad 
M = L^{-1} = \Phi\circ F^{-1}.
$$
Via Taylor-approximation and $\tilde\alpha:=1/F'(0)>0$, $\tilde\beta:= F''(0)>0$ having the same values as in the proof of Lemma~\ref{expand_plus} we obtain as $s\to 2\pi+$
\begin{align*}
F^{-1}(s) &= \tilde\alpha(s-2\pi)(1+O(s-2\pi)), \\
(F^{-1})'(s) &= \tilde\alpha(1+O(s-2\pi)), \\
(F^{-1})''(s) &= -\tilde\beta\tilde\alpha^3(1+O(s-2\pi)).
\end{align*}
Hence as $s\to 2\pi+$ we obtain
$$
M(s) = \Phi(F^{-1}(s)) = \alpha(s-2\pi)^\frac{2}{p-1}(1+O(s-2\pi))
$$
for $\alpha=  \left(\frac{p+1}{2}\right)^\frac{2}{p-1} \tilde\alpha^\frac{2}{p-1} >0$, 
\begin{align*}
M'(s) & = \Phi'(F^{-1}(s)) (F^{-1})'(s) \\
& = \left(\frac{p+1}{2}\right)^\frac{2}{p-1}\left( \frac{2}{p-1} (F^{-1}(s))^\frac{3-p}{p-1}-\frac{p+1}{p-1}(F^{-1}(s))^\frac{2}{p-1}\right)\tilde\alpha(1+O(s-2\pi)) \\
& = \frac{2\alpha}{p-1}(s-2\pi)^\frac{3-p}{p-1}(1+O(s-2\pi)) 
\end{align*}
and 
\begin{align*}
M''(s) =&  \Phi''(F^{-1}(s)) \Bigl((F^{-1})'(s)\Bigr)^2 + \Phi'(F^{-1}(s)) (F^{-1})''(s) \\
=& \left(\frac{p+1}{2}\right)^\frac{2}{p-1}\left( \frac{2(3-p)}{(p-1)^2} (F^{-1}(s))^\frac{4-2p}{p-1} - \frac{2(p+1)}{(p-1)^2}(F^{-1}(s))^\frac{3-p}{p-1}\right)\tilde\alpha^2(1+O(s-2\pi))\\
& + O\bigl((s-2\pi)^\frac{3-p}{p-1}\bigr) \\
=& \frac{2\alpha(3-p)}{(p-1)^2}(s-2\pi)^\frac{4-2p}{p-1}(1+O(s-2\pi)).
\end{align*}
As before, the expansions for $\sqrt{M}, {\sqrt{M}\,}' = \frac{M'}{2\sqrt{M}}$ and ${\sqrt{M}\,}'' = \frac{1}{2 M^{3/2}}\left(M''M-\frac{1}{2} (M')^2\right)$ follow directly from the expansions for $M, M', M''$. 
\end{proof}

\section*{Acknowledgments}
We gratefully acknowledge financial support by the Deutsche Forschungs\-gemeinschaft (DFG) through CRC 1173.

%%%%%%%%%%%%%%%%%%%%%%%%%%%%%%%%%%%%%%%%%%%%%%%%%%%%%%%%%%%%%%%%%%%%%%%%%%%%%%%%
%%%%%%%%%%%
% Bibliography
%%%%%%%%%%%%%%%%%%%%%%%%%%%%%%%%%%%%%%%%%%%%%%%%%%%%%%%%%%%%%%%%%%%%%%%%%%%%%%%%
%%%%%%%%%%%
\bibliographystyle{plain}	
\bibliography{bibliography}

\end{document}